\documentclass[reqno,12pt]{amsart}
\usepackage{amsmath}
\usepackage{amssymb}

\numberwithin{equation}{section}

\newcommand\NoBlackBoxes{\global\overfullrule0pt}
\NoBlackBoxes

\newtheorem{definition}{Definition}[section]
\newtheorem{theorem}[definition]{Theorem}
\newtheorem{lemma}[definition]{Lemma}
\newtheorem{rem}[definition]{Remark}

\newcommand{\N}{{\mathbb N}}
\newcommand{\R}{{\mathbb R}}

\newcommand{\E}{{\mathbb E}}

\newcommand{\cov}{\mathrm{Cov}}
\newcommand{\la}{\lambda}

\newcommand{\X}{\textbf{X}}
\newcommand{\tr}{\mathrm{tr}}
\newcommand{\nc}{\mathcal{N}\mathcal{P}\mathcal{P}(k)}
\newcommand{\bpi}{\boldsymbol{\pi}}

\title{The semicircle law for matrices with independent diagonals}

\author[Olga Friesen]{Olga Friesen}
\address[Olga Friesen]{Westf\"alische Wilhelms-Universit\"at M\"unster,
Fachbereich Mathematik,
Einsteinstra\ss e 62, 48149 M\"unster, Germany}
\email[Olga Friesen]{olga.friesen@uni-muenster.de}

\author[Matthias L\"owe]{Matthias L\"owe}
\address[Matthias L\"owe]{Westf\"alische Wilhelms-Universit\"at M\"unster,
Fachbereich Mathematik,
Einsteinstra\ss e 62, 48149 M\"unster, Germany}
\email[Matthias L\"owe]{maloewe@math.uni-muenster.de}

\date{\today}


\begin{document}
\begin{abstract}
We investigate the spectral distribution of random matrix ensembles with correlated entries. We consider symmetric matrices with real valued entries and
stochastically independent diagonals. Along the diagonals the entries may be correlated. We show that under sufficiently nice moment conditions
the empirical eigenvalue distribution converges almost surely weakly to the semi-circle law.

\end{abstract}

\maketitle

\bibliographystyle{alpha}

\section{Introduction}

Large-dimensional random matrices are, among others, of interest in statistics and in theoretical physics, in particular when studying the properties of atoms with heavy
nuclei. One of the most interesting and best studied questions, has been to investigate the properties of the eigenvalues of random matrices. For example, Wigner, in his
seminal paper \cite{Wigner} showed that the spectral distribution of symmetric or Hermitian random matrices with independent Gaussian entries, otherwise, under appropriate
scaling converges to the semi-circle law. This was generalized by Arnold \cite{Arnold} to the situation of symmetric or Hermitian random matrices filled with independent
and identically distributed (i.i.d.) random variables with sufficiently many moments. Other generalizations of Wigner's semi-circle law concern matrix ensembles with entries
drawn according to weighted Haar measures on classical (e.g., orthogonal, unitary, symplectic) groups. Such results are particularly interesting, since such random matrices
also play a major role in non-commutative probability (see e.g. \cite{alice_stflour}); other applications are in graph theory, combinatorics or algebra.

This note addresses a question that is much in the spirit of Arnold's generalization of the semi-circle law. Even though a couple of random matrix models include situations
with stochastically correlated entries (see especially \cite{brycdembo}, where the case of random Toeplitz and Hankel matrices is treated), the dependencies are not very
natural from a stochastic point of view. A generic way to construct random matrices with dependent entries could be to consider a two dimensional (stationary) random field
indexed by $\mathbb{Z}^d$ with correlations that decay with the distance of the indices and to take an $n \times n$ block as entries for a random $n \times n$ matrix.

The present note is a first step to study the asymptotic eigenvalue distribution of such matrix ensembles. Here we will deviate from the independence assumption by
considering (real) random fields with entries that may be dependent on each diagonal, but with stochastically independent diagonals. For such matrices we will prove a semi-circle law.
The setup may look at first glance a bit more artificial than a situation where the matrices are filled with row- or columnwise independent random variables (e.g.
with row- or columnwise independent Markov chains). Note, however, that in order guarantee for real eigenvalues we will need to restrict ourselves to symmetric random matrices.
This would imply that a matrix with rowwise independent entries above the diagonal has columnwise independent entries below it. Not only is this a rather strange setup, also can one
see from simulations that their asymptotic eigenvalue distribution is probably not the semi-circle law.

It also should be mentioned that a similar situation has been studied by Khorunzhy and Pastur in \cite{KhorunzhyPastur94}. They consider the eigenvalue distribution of so called deformed Wigner ensembles that consist of matrices which can be written as a sum of Wigner matrix (a symmetric matrix with independent entries above the diagonal) and a deterministic matrix. It is proven that in this situation the
empirical eigenvalue density converges in probability to a non-random limit. This setup, yet similar, is different from ours.

The rest of the note is organized as follows. In the second section we will formalize the situation we want to consider and state our main result. Section 3 is devoted to the
proof, that is based on a moment method. Section 4 contains some examples.

\section{Main Results}

In this section we will state our main theorem, a semi-circle law for symmetric random matrices with independent diagonals (for a precise formulation see Theorem
\ref{main} below). The limit law for their empirical eigenvalue distribution is the semi-circle distribution. Its density is given by
$$
f(x)=\left\{\begin{array}{ll}
\frac 1 {2\pi} \sqrt {4-x^2} & \mbox{if } -2 \le x \le 2\\
0  & \mbox{otherwise.}
\end{array} \right.
$$

We want to consider the following setup: Let $\left\{a(p,q), 1\leq p\leq q< \infty\right\}$ be a real valued random field.
For any $n\in \mathbb{N}$, define the symmetric random $n\times n$ matrix $\X_n$ by
\begin{equation*}
 \X_n (q,p) = \X_n (p,q) = \frac{1}{\sqrt{n}} a(p,q), \qquad 1\leq p\leq q\leq n,
\end{equation*}
We will have to impose the following conditions on $\X_n$: \\

\begin{enumerate}
 \item[(C1)] $\mathbb{E}\left[a(p,q)\right]=0$, $\mathbb{E}\left[a(p,q)^{2}\right] = 1$ and
\begin{equation*}
 m_k:=\sup_{n\in\mathbb{N}} \max_{1\leq p\leq q\leq n} \mathbb{E}\left[\left|a(p,q)\right|^{k}\right] < \infty, \quad k\in\mathbb{N}.
\end{equation*}
 \item[(C2)] the diagonals of $\X_n$, i.e. the families $\left\{a(p,p+r), p\in\mathbb{N}\right\}$, $r\in\mathbb{N}_0$, are independent,
 \item[(C3)] the covariance of two entries on the same diagonal depends only on their distance, i.e. for any $\tau\in\mathbb{N}_0$ we can define
	\begin{equation*}
       	\cov(\tau) := \cov(a(p,q),a(p+\tau,q+\tau)), \qquad p,q\in\mathbb{N},
	\end{equation*}
 \item[(C4)] the entries on the diagonals have a quickly decaying dependency structure, which will be expressed in terms of the condition
	\begin{equation*}
	 \sum_{\tau=0}^{\infty} \left|\cov(\tau)\right| < \infty.
	\end{equation*}
\end{enumerate}

We will denote the (real) eigenvalues of $ \textbf{X}_{n}$ by $\lambda_1^{(n)} \le \lambda_2^{(n)} \le \ldots \lambda_n^{(n)}$. Let $\mu_n$ be the empirical eigenvalue distribution, i.e.
\begin{equation*}
\mu_n = \frac{1}{n} \sum_{k=1}^n \delta_{\lambda_k^{(n)}}.
\end{equation*}

With these notations we are able to formulate the
central result of this note.

\begin{theorem}
Assume that the symmetric random matrix $\X_n$ as defined above satisfies the conditions $(C1)$, $(C2)$, $(C3)$ and $(C4)$. Then, with probability $1$, the empirical spectral distribution of $\textbf{X}_n$ converges weakly to the standard semi-circle distribution, i.e.
\begin{equation*}
\mu_n \Rightarrow \mu \qquad \mbox{ as } n \to \infty
\end{equation*}
both, in expectation and $\mathbb{P}-\mbox{almost surely}$.
Here ''$\Rightarrow$'' denotes weak convergence.
\label{main}
\end{theorem}

\begin{rem}
\normalfont{
 Note that in order for the semi-circle law to hold, it is not possible to renounce condition $(C4)$ without any replacement. Consider for example a Toeplitz matrix, that is a Hermitian matrix with identical entries on each diagonal. For such a matrix, we clearly have
\begin{equation*}
 \sum_{\tau=0}^{\infty} \left|\cov(\tau)\right| = \infty.
\end{equation*}
Indeed, it was shown in \cite{brycdembo} that the empirical distribution of a sequence of Toeplitz matrices tends with probability $1$ to a nonrandom probability measure with unbounded support.
}
\end{rem}

\section{Proof of Theorem \ref{main}}
We want to resort to the method of moments to prove Theorem \ref{main} (this method has been applied in similar situations in \cite{Arnold} or \cite{Schenker_Schulz-Baldes}, among (many) others). To this end, let $Y$ be distributed according to the semi-circle distribution. For the proof of the theorem it will be important to notice that the moments of $Y$ are given by
\begin{equation}
\E(Y^{k})=\begin{cases} 0, & \text{if} \ k \ \text{is odd}, \\ C_{\frac{k}{2}}, & \text{if} \ k \ \text{is even}, \end{cases}
\end{equation}
where $C_{\frac{k}{2}} = \frac{k!}{\frac{k}{2}!\left(\frac{k}{2}+1\right)!}$ denote the Catalan numbers. Since these moments determine the semicircle distribution uniquely, the weak convergence of the expected empirical distribution will follow from the relation
\begin{equation*}
 \lim_{n\to\infty} \frac{1}{n} \E\left[\tr\left(\X_{n}^{k}\right)\right] = \begin{cases} 0, & \text{if} \ k \ \text{is odd}, \\ C_{\frac{k}{2}}, & \text{if} \ k \ \text{is even,} \end{cases}
\end{equation*}
where $\tr(\cdot)$ denotes the trace operator. The first part of the proof is to verify this convergence.\\

To start with, consider the set $\mathcal{T}_n(k)$ of $k$-tuples of consistent pairs, that is elements of the form $\left(P_1,\ldots,P_k\right)$ with $P_j = (p_j,q_j) \in \left\{1,\ldots,n\right\}^2$ satisfying $q_j = p_{j+1}$ for any $j=1,\ldots,k$, where $k+1$ is identified with $1$. Then, we have
\begin{equation*}
 \frac{1}{n} \E\left[\tr\left(\X_{n}^{k}\right)\right] = \frac{1}{n^{1+\frac{k}{2}}} \sum_{\left(P_1,\ldots,P_k\right)\in\mathcal{T}_n(k)} \E\left[a(P_1)\cdot \ldots \cdot a(P_k)\right].
\end{equation*}
Further, define $\mathcal{P}(k)$ to be the set of all partitions $\pi$ of $\left\{1,\ldots,k\right\}$. Any partition $\pi$ induces an equivalence relation $\sim_\pi$ on $\left\{1,\ldots,k\right\}$ by
\begin{equation*}
 i\sim_\pi j \quad :\Longleftrightarrow \quad \text{$i$ and $j$ belong to the same set of the partition} \ \pi.
\end{equation*}

We say that an element $\left(P_1,\ldots,P_k\right)\in\mathcal{T}_n(k)$ is a $\pi$ consistent sequence if
\begin{equation*}
 \left|p_i - q_i\right| = \left|p_j - q_j\right| \quad \Longleftrightarrow \quad i\sim_{\pi} j.
\end{equation*}

Due to condition $(C2)$, this implies that $a(P_{i_1}),\ldots,a(P_{i_l})$ are stochastically independent if $i_1,\ldots,i_l$ belong to $l$ different blocks of $\pi$. The set of all $\pi$ consistent sequences $\left(P_1,\ldots,P_k\right)\in\mathcal{T}_n(k)$ is denoted by $S_n(\pi)$. Thus, we can write
\begin{equation*}
 \frac{1}{n} \E\left[\tr\left(\X_{n}^{k}\right)\right] = \frac{1}{n^{1+\frac{k}{2}}} \sum_{\pi \in \mathcal{P}(k)} \sum_{\left(P_1,\ldots,P_k\right)\in S_n(\pi)} \E\left[a(P_1)\cdot \ldots \cdot a(P_k)\right].
\label{sum}
\end{equation*}
Now fix a $k\in\mathbb{N}$. For any $\pi\in\mathcal{P}(k)$ let $\# \pi$ denote the number of equivalence classes of $\pi$. We distinguish different cases.\\

\noindent \textbf{First case:} $\quad \# \pi > \frac{k}{2}$ \\

Since $\pi$ is a partition of $\left\{1,\ldots,k\right\}$, there is at least one equivalence class with a single element $l$. Consequently, for any sequence $\left(P_1,\ldots,P_k\right)\in S_n(\pi)$ we have
\begin{equation*}
 \E\left[a(P_1)\cdot \ldots \cdot a(P_k)\right] = \E\Big[\prod_{i\neq l}a(P_i)\Big] \cdot \E\left[a(P_l)\right] = 0,
\end{equation*}
due to the independence of elements in different equivalence classes.

Hence, we obtain
\begin{equation*}
 \frac{1}{n} \E\left[\tr\left(\X_{n}^{k}\right)\right] = \frac{1}{n^{1+\frac{k}{2}}} \sum_{\underset{\# \pi \leq \frac{k}{2}}{\pi \in \mathcal{P}(k),}} \sum_{\left(P_1,\ldots,P_k\right)\in S_n(\pi)} \E\left[a(P_1)\cdot \ldots \cdot a(P_k)\right].
\end{equation*}

\noindent \textbf{Second case:} $\quad r:= \# \pi < \frac{k}{2}$ \\

We need to calculate $\# S_n(\pi)$. To fix an element $\left(P_1,\ldots,P_k\right)\in S_n(\pi)$, we first choose the pair $P_1 = (p_1,q_1)$. There are at most $n$ possibilities to assign a value to $p_1$ and another $n$ possibilities for $q_1$. To fix $P_2 = (p_2,q_2)$, note that the consistency of the pairs implies $p_2 = q_1$. If now $1\sim_\pi 2$, the condition $\left|p_1 - q_1\right| = \left|p_2 - q_2\right|$ allows at most two choices for $q_2$. Otherwise, if $1\not\sim_\pi 2$, we have at most $n$ possibilities. We now proceed sequentially to determine the remaining pairs. When arriving at some index $i$, we check whether $i$ is equivalent to any preceding index $1,\ldots,i-1$. If this is the case, then we have at most two choices for $P_i$ and otherwise, we have $n$. Since there are exactly $r$ different equivalence classes, we can conclude that
\begin{equation*}
 \# S_n(\pi) \leq n^2 \cdot n^{r-1} \cdot 2^{k-r} \leq C \cdot n^{r+1}
\end{equation*}
with a constant $C=C(r,k)$ depending on $r$ and $k$.

Now the uniform boundedness of the moments and the H\"{o}lder inequality together imply that for any sequence $(P_1,\ldots,P_k)$,
\begin{equation}
 \left|\E\left[a(P_1)\cdot \ldots \cdot a(P_k)\right] \right| \leq \left[\E\left|a(P_1)\right|^{k}\right]^{\frac{1}{k}} \cdot \ldots \cdot \left[\E\left|a(P_k)\right|^{k}\right]^{\frac{1}{k}} \leq m_k.
\label{holder}
\end{equation}
Consequently, taking account of the relation $r< \frac{k}{2}$, we get
\begin{equation*}
 \frac{1}{n^{1+\frac{k}{2}}} \sum_{\underset{\# \pi < \frac{k}{2}}{\pi \in \mathcal{P}(k),}} \sum_{\left(P_1,\ldots,P_k\right)\in S_n(\pi)} \left|\E\left[a(P_1)\cdot \ldots \cdot a(P_k)\right] \right| \leq C \cdot \frac{1}{n^{1+\frac{k}{2}}} \cdot n^{r+1} = o(1).
\end{equation*} \\

\noindent Combining the calculations in the first and the second case, we can conclude that
\begin{equation*}
 \lim_{n\to\infty} \frac{1}{n} \E\left[\tr\left(\X_{n}^{k}\right)\right] = \lim_{n\to\infty} \frac{1}{n^{1+\frac{k}{2}}} \sum_{\underset{\# \pi = \frac{k}{2}}{\pi \in \mathcal{P}(k),}} \sum_{\left(P_1,\ldots,P_k\right)\in S_n(\pi)} \E\left[a(P_1)\cdot \ldots \cdot a(P_k)\right],
\end{equation*}
if the limits exist. \\

\noindent Now consider the case where $k$ is \textbf{odd}. Since then the condition $\# \pi = \frac{k}{2}$ cannot be satisfied, the considerations above immediately yield
\begin{equation*}
 \lim_{n\to\infty} \frac{1}{n} \E\left[\tr\left(\X_{n}^{k}\right)\right] = 0.
\end{equation*} \\

\noindent It remains to cope with \textbf{even} $k$. Denote by $\mathcal{P}\mathcal{P}(k)\subset \mathcal{P}(k)$
the set of all pair partitions of $\left\{1,\ldots,k\right\}$. In particular, $\# \pi = \frac{k}{2}$ for any $\pi \in \mathcal{P}\mathcal{P}(k)$.
On the other hand, if $\#\pi = \frac{k}{2}$ but $\pi \notin \mathcal{P}\mathcal{P}(k)$, we can conclude that $\pi$ has at least one equivalence class with a
single element and hence, as in the first case, the expectation corresponding to the $\pi$ consistent sequences will become zero. Consequently,
\begin{equation*}
 \lim_{n\to\infty} \frac{1}{n} \E\left[\tr\left(\X_{n}^{k}\right)\right] = \lim_{n\to\infty} \frac{1}{n^{1+\frac{k}{2}}} \sum_{\pi\in\mathcal{P}\mathcal{P}(k)} \sum_{\left(P_1,\ldots,P_k\right)\in S_n(\pi)} \E\left[a(P_1)\cdot \ldots \cdot a(P_k)\right],
\end{equation*}
if the limits exist. We have now reduced the original set $\mathcal{P}(k)$ to the subset $\mathcal{P}\mathcal{P}(k)$. Next we want to fix a $\pi\in\mathcal{P}\mathcal{P}(k)$ and cope with the set $S_n(\pi)$.

\begin{lemma}[cf. \cite{brycdembo}, Proposition 4.4.]
 Let $S_{n}^{*}(\pi) \subseteq S_n(\pi)$ denote the set of $\pi$ consistent sequences $(P_1,\ldots,P_k)$ satisfying
\begin{equation*}
 i\sim_\pi j \quad \Longrightarrow \quad q_i - p_i = p_j - q_j
\end{equation*}
for all $i\neq j$. Then, we have
\begin{equation*}
 \# \left(S_n(\pi)\backslash S_{n}^{*}(\pi)\right) = o\left(n^{1+\frac{k}{2}}\right).
\end{equation*}
\label{snstar}
\end{lemma}

\begin{proof}
 We call a pair $(P_i,P_j)$ with $i\sim_\pi j$, $i\neq j$, positive if $q_i-p_i = q_j - p_j > 0$ and negative if $q_i - p_i = q_j - p_j < 0$. Since $\sum_{i=1}^k q_i - p_i = 0$ by consistency, the existence of a negative pair implies the existence of a positive one. Thus, we can assume that any sequence $(P_1,\ldots,P_k) \in S_n(\pi)\backslash S_{n}^{*}(\pi)$ contains a positive pair $(P_l,P_m)$. To fix such a sequence, we first determine the positions of $l$ and $m$ and then, we fix the signs of the remaining differences $q_i - p_i$. The number of possibilities to accomplish that depends only on $k$ and not on $n$. Now we choose one of $n$ possible values for $p_l$. In a next step, we fix the values of the differences $\left|q_i - p_i\right|$ for all $P_i$ except for $P_l$ and $P_m$. Since in each case two pairs are equivalent, i.e. the difference of the indices is equal, we have $n^{\frac{k}{2}-1}$ possibilities for that. Then, $\sum_{i=1}^k q_i - p_i = 0$ implies that
\begin{equation*}
 0 < 2(q_l - p_l) = q_l - p_l + q_m - p_m = \sum_{\underset{i\neq l,m}{i=1,}}^k p_i - q_i.
\end{equation*}
Since we have already chosen the signs of the differences as well as their absolute values, we know the value of the sum on the right hand side. Hence, the difference $q_l - p_l = q_m - p_m$ is fixed. We now have the index $p_l$, all differences $\left|q_i - p_i\right|, i\in\left\{1,\ldots,k\right\}$, and their signs. Thus, we can start at $P_l$ and go systematically through the whole sequence $(P_1,\ldots,P_k)$ to see that it is uniquely determined. Consequently, our considerations lead to
\begin{equation*}
 \# \left(S_n(\pi)\backslash S_{n}^{*}(\pi)\right) \leq C \cdot n^{\frac{k}{2}} = o\left(n^{1+\frac{k}{2}}\right).
\end{equation*}

\end{proof}

\noindent As a consequence of Lemma~\ref{snstar} and relation \eqref{holder}, we obtain

\begin{equation*}
 \lim_{n\to\infty} \frac{1}{n} \E\left[\tr\left(\X_{n}^{k}\right)\right] = \lim_{n\to\infty} \frac{1}{n^{1+\frac{k}{2}}} \sum_{\pi\in\mathcal{P}\mathcal{P}(k)} \sum_{\left(P_1,\ldots,P_k\right)\in S_{n}^{*}(\pi)} \E\left[a(P_1)\cdot \ldots \cdot a(P_k)\right],
\end{equation*}

\noindent if the limits exist. \\

\noindent We call a pair partition $\pi \in \mathcal{P}\mathcal{P}(k)$ \textbf{crossing} if there are indices $i<j<l<m$ with $i\sim_\pi l$ and $j\sim_\pi m$. Otherwise, we call $\pi$ \textbf{non-crossing}. The set of all non-crossing pair partitions is denoted by $\nc$.

\begin{lemma}
 For any crossing $\pi \in \mathcal{P}\mathcal{P}(k) \backslash \nc$, we have
\begin{equation*}
 \sum_{\left(P_1,\ldots,P_k\right)\in S_{n}^{*}(\pi)} \E\left[a(P_1)\cdot \ldots \cdot a(P_k)\right] = o\left(n^{\frac{k}{2}+1}\right).
\end{equation*}
\label{crossing}
\end{lemma}

\begin{proof}
Let $\pi$ be crossing and consider a sequence $\left(P_1,\ldots,P_k\right)\in S_{n}^{*}(\pi)$. Note that if there is an $l\in\left\{1,\ldots,k\right\}$ with $l\sim_\pi l+1$, where $k+1$ is identified with $1$, we immediately have
\begin{equation*}
 a(P_l) = a(P_{l+1}),
\end{equation*}
since $q_l = p_{l+1}$ by consistency and then $p_l = q_{l+1}$ by definition of $S_{n}^{*}(\pi)$. In particular,
\begin{equation*}
 \E\left[a(P_l) \cdot a(P_{l+1})\right] = 1.
\end{equation*}
The sequence $\left(P_1,\ldots, P_{l-1},P_{l+2}, \ldots,P_k\right)$ is still consistent because of the relation $q_{l-1} = p_l = q_{l+1} = p_{l+2}$. Since there are at most $n$ choices for $q_l = p_{l+1}$, it follows
\begin{equation*}
 \# S_{n}^{*}(\pi) \leq n \cdot \# S_{n}^{*}(\pi^{(1)}),
\end{equation*}
where $\pi^{(1)}\in \mathcal{P}\mathcal{P}(k-2) \backslash \mathcal{N}\mathcal{P}\mathcal{P}(k-2)$ is the pair partition induced by $\pi$ after eliminating the indices $l$ and $l+1$. Let $r$ denote the maximum number of pairs of indices that can be eliminated in this way. Since $\pi$ is crossing, there are at least two pairs left and hence, $r\leq\frac{k}{2}-2$. By induction, we conclude that
\begin{equation*}
 \# S_{n}^{*}(\pi) \leq n^r \cdot \# S_{n}^{*}(\pi^{(r)}),
\end{equation*}
where now $\pi^{(r)}\in \mathcal{P}\mathcal{P}(k-2r) \backslash \mathcal{N}\mathcal{P}\mathcal{P}(k-2r)$ is the still crossing pair partition induced by $\pi$. Thus, we so far have
\begin{align}
\begin{split}
& \sum_{\left(P_1,\ldots,P_k\right)\in S_{n}^{*}(\pi)} \left|\E\left[a(P_1)\cdot \ldots \cdot a(P_k)\right]\right| \\
& \qquad \qquad \qquad \qquad \leq n^r \sum_{(P^{(r)}_1,\ldots,P^{(r)}_{k-2r})\in S_{n}^{*}(\pi^{(r)})} \left|\E\left[a(P^{(r)}_1)\cdot \ldots \cdot a(P^{(r)}_k)\right]\right|.
\end{split}
\label{estimate}
\end{align}
Choose $i\sim_{\pi^{(r)}} i+j$ such that $j$ is minimal. We want to count the number of sequences $(P_1^{(r)},\ldots,P_{k-2r}^{(r)})\in S_{n}^{*}(\pi^{(r)})$ given that $p^{(r)}_i$ and $q^{(r)}_{i+j}$ are fixed. Therefore, we start with choosing one of $n$ possible values for $q^{(r)}_i$. But then, we can also deduce the value of
\begin{equation*}
 p^{(r)}_{i+j} = q^{(r)}_i - p^{(r)}_i + q^{(r)}_{i+j}.
\end{equation*}
Since $j$ is minimal, any element in $\left\{i+1,\ldots,i+j-1\right\}$ is equivalent to some element outside the set $\left\{i,\ldots,i+j\right\}$. There are $n$ possibilities to fix $P^{(r)}_{i+1}$ as $p^{(r)}_{i+1}=q^{(r)}_i$ is already fixed. Proceeding sequentially, we have $n$ possibilities for the choice of any pair $P^{(r)}_l$ with $l\in \left\{i+2,\ldots,i+j-2\right\}$ and there is only one choice for $P^{(r)}_{i+j-1}$ since $q^{(r)}_{i+j-1}=p^{(r)}_{i+j}$ is already chosen. For any other pair that has not yet been fixed, there are at most $n$ possibilities if it is not equivalent to one pair that has already been chosen. Otherwise, there is only one possibility. Hence, assuming that the elements $p^{(r)}_i$ and $q^{(r)}_{i+j}$ are fixed, we have at most
\begin{equation*}
 n \cdot n^{j-2} \cdot n^{\frac{k}{2}-r-j} = n^{\frac{k}{2}-r-1}
\end{equation*}
possibilities to choose the rest of the sequence $(P^{(r)}_1,\ldots,P^{(r)}_{k-2r})\in S_{n}^{*}(\pi^{(r)})$. Consequently, estimating the term in \eqref{estimate} further, we obtain
\begin{align*}
\sum_{\left(P_1,\ldots,P_k\right)\in S_{n}^{*}(\pi)} \left|\E\left[a(P_1)\cdot \ldots \cdot a(P_k)\right]\right| & \leq n^{\frac{k}{2}-1} \sum_{p^{(r)}_i,q^{(r)}_{i+j}=1}^{n} | \cov(|q^{(r)}_{i+j}-p^{(r)}_i|) | \\
& \leq C \cdot n^{\frac{k}{2}} \sum_{\tau=0}^{n-1} \left|\cov(\tau)\right| = o\left(n^{1+\frac{k}{2}}\right),
\end{align*}
since $\sum_{\tau=0}^{\infty} \left|\cov(\tau)\right| < \infty$ by condition $(C4)$.

\end{proof}

\noindent Lemma~\ref{crossing} now guarantees that we need to consider only non-crossing pair partitions, that is
\begin{equation*}
 \lim_{n\to\infty} \frac{1}{n} \E\left[\tr\left(\X_{n}^{k}\right)\right] = \lim_{n\to\infty} \frac{1}{n^{1+\frac{k}{2}}} \sum_{\pi\in\nc} \sum_{\left(P_1,\ldots,P_k\right)\in S_{n}^{*}(\pi)} \E\left[a(P_1)\cdot \ldots \cdot a(P_k)\right],
\end{equation*}
if the limits exist.

\begin{lemma}
 Let $\pi \in \nc$. For any $\left(P_1,\ldots,P_k\right)\in S_{n}^{*}(\pi)$, we have
\begin{equation*}
 \E\left[a(P_1)\cdot\ldots\cdot a(P_k)\right] = 1.
\end{equation*}
\label{le1}
\end{lemma}

\begin{proof}
 Let $l<m$ with $m\sim_\pi l$. Since $\pi$ is non-crossing, the number $l-m-1$ of elements between $l$ and $m$ must be even. In particular, there is $l\leq i< j\leq m$ with $i\sim_\pi j$ and $j=i+1$. By the properties of $S_{n}^{*}$, we have $a(P_i)=a(P_j)$, and the sequence $\left(P_1,\ldots, P_l,\ldots,P_{i-1},P_{i+2},\ldots,P_m,\ldots,P_k\right)$ is still consistent. Applying this argument successively, all pairs between $l$ and $m$ vanish and we see that the sequence $\left(P_1,\ldots,P_l,P_m,\ldots,P_k\right)$ is consistent, that is $q_l=p_m$. Then, the identity $p_l=q_m$ also holds. In particular, $a(P_l)=a(P_m)$. Since $l,m$ have been chosen arbitrarily, we obtain
\begin{equation*}
 \E\left[a(P_1)\cdot\ldots\cdot a(P_k)\right] = \prod_{\stackrel{l< m}{l\sim_\pi m}} \E\left[a(P_l)\cdot a(P_m)\right] = 1.
\end{equation*}
\end{proof}

\noindent It remains to verify

\begin{lemma}
 For any $\pi\in\nc$, we have
\begin{equation*}
 \lim_{n\to\infty} \frac{\# S_{n}^{*}(\pi)}{n^{\frac{k}{2}+1}} = 1.
\end{equation*}
\label{le2}
\end{lemma}

\begin{proof}
 To calculate the number of elements in $S_{n}^{*}(\pi)$, first choose $P_1$. There are $n^2$ possibilities for that choice. If $1\sim_\pi 2$, then $P_2$ is uniquely determined since $p_2 = q_1$ and by definition of $S_{n}^{*}(\pi)$, $q_2 = p_1$. If $1\not\sim_\pi 2$, then there are $n-1$ possibilities to fix $P_2$. Proceeding in the same way, we see that if $i \in \left\{2,\ldots,k\right\}$ is equivalent to some element in $\left\{1,\ldots,i-1\right\}$, there is always only one value $P_i$ can take. Otherwise there are asymptotically $n$ choices. The latter case will occur exactly $\frac{k}{2}-1$ times. In conclusion,
\begin{equation*}
 \# S_{n}^{*}(\pi) \sim n^2 \cdot n^{\frac{k}{2}-1} = n^{1+\frac{k}{2}}.
\end{equation*}

\end{proof}

\noindent Lemma~\ref{le1} and Lemma~\ref{le2} now provide that
\begin{align*}
 \lim_{n\to\infty} \frac{1}{n} \E\left[\tr\left(\X_{n}^{k}\right)\right] &= \lim_{n\to\infty} \frac{1}{n^{1+\frac{k}{2}}} \sum_{\pi\in\nc} \sum_{\left(P_1,\ldots,P_k\right)\in S_{n}^{*}(\pi)} \E\left[a(P_1)\cdot \ldots \cdot a(P_k)\right] \\
&= \lim_{n\to\infty} \frac{1}{n^{1+\frac{k}{2}}} \sum_{\pi\in\nc} \# S_{n}^{*}(\pi) = \# \nc.
\end{align*}

\noindent Since the number of non-crossing pair partitions $\# \nc$ equals exactly the Catalan number $C_{\frac{k}{2}}$, we can conclude that the expected empirical spectral distribution of $\X_n$ tends to the semi-circle law. This is the asserted convergence in expectation. \\

It remains to deduce almost sure convergence. Therefore, we want to follow the ideas of \cite{brycdembo}. To this end, we need

\begin{lemma}
 Suppose the conditions of Theorem~\ref{main} hold. Then, for any $k,n \in\mathbb{N}$,

\begin{equation*}
 \E\left[\left(\tr\left(\X_{n}^{k}\right) - \E\left[\tr \left(\X_{n}^{k}\right)\right]\right)^4\right] \leq C \cdot n^{2}.
\end{equation*}

\label{lemma}
\end{lemma}

\begin{proof}

Fix $k,n \in\mathbb{N}$. Using the notation
\begin{equation*}
 P = (P_1,\ldots,P_k) = ( (p_1,q_1), \ldots, (p_k,q_k) ), \qquad a (P) = a(P_{1})\cdot \ldots \cdot a(P_{k}),
\end{equation*}

\noindent we have that

\begin{multline}
 \E\left[\left(\tr\left(\X_{n}^{k}\right) - \E\left[\tr \left(\X_{n}^{k}\right)\right]\right)^4\right] \\
= \frac{1}{n^{2k}} \sum_{\pi^{(1)},\ldots,\pi^{(4)} \in \mathcal{P}(k)} \sum_{P^{(i)}\in S_n\left(\pi^{(i)}\right), i=1,\ldots,4} \E\Big[\prod_{j=1}^{4} \left(a (P^{(j)}) - \E\left[a (P^{(j)})\right]\right)\Big].
\label{eq1}
\end{multline}

Now consider a partition $\boldsymbol{\pi}$ of $\left\{1,\ldots,4k\right\}$. We say that a sequence $(P^{(1)},\ldots,P^{(4)})$ is $\bpi$ consistent if each $P^{(i)}, i=1,\ldots,4$, is a consistent sequence and

\begin{equation*}
\big|q_{l}^{(i)} - p_{l}^{(i)}\big| \ = \ \left|q_{m}^{(j)} - p_{m}^{(j)}\right| \quad \Longleftrightarrow \quad l + (i-1) k \ \sim_{\boldsymbol{\pi}} \ m + (j-1) k.
\end{equation*}
Let $\mathcal{S}_n (\bpi)$ denote the set of all $\bpi$ consistent sequences with entries in $\left\{1,\ldots,n\right\}$. Then, \eqref{eq1} becomes
\begin{multline}
 \E\left[\left(\tr\X_{n}^{k} - \E\left[\tr \X_{n}^{k}\right]\right)^4\right] \\
= \frac{1}{n^{2k}} \sum_{\bpi\in\mathcal{P}(4k)} \sum_{(P^{(1)}, \ldots, P^{(4)})\in \mathcal{S}_n\left(\bpi\right)} \E\Big[\prod_{j=1}^{4} \left(a (P^{(j)}) - \E\left[a (P^{(j)})\right]\right)\Big].
\label{eq2}
\end{multline}

We want to analyze the expectation on the right hand side. Therefore, fix a $\bpi \in \mathcal{P}(4k)$. We call $\boldsymbol{\pi}$ a matched partition if
\begin{enumerate}
 \item any equivalence class of $\bpi$ contains at least two elements and
 \item for any $i\in\left\{1,\ldots,4\right\}$ there is a $j\neq i$ and $l,m\in\left\{1,\ldots,k\right\}$ with
	\begin{equation*}
	 l + (i-1) k \ \sim_{\boldsymbol{\pi}} \ m + (j-1) k.
	\end{equation*}
\end{enumerate}
In case $\boldsymbol{\pi}$ is not matched, we can conclude that
\begin{align*}
\sum_{(P^{(1)}, \ldots, P^{(4)})\in \mathcal{S}_n\left(\bpi\right)} \E\Big[\prod_{j=1}^{4} \left(a (P^{(j)}) - \E\left[a (P^{(j)})\right]\right)\Big] = 0.
\end{align*}

\noindent Thus, we only have to consider matched partitions to evaluate the sum in \eqref{eq2}. Let $\boldsymbol{\pi}$ be such a partition and denote by $r = \#\bpi$ the number of equivalence classes of $\bpi$. Note that condition $(i)$ implies $r\leq 2k$. To count all $\bpi$ consistent sequences $(P^{(1)},\ldots,P^{(4)})$, we first choose one of at most $n^r$ possibilities to fix the $r$ different equivalence classes. Afterwards, we fix the elements $p_1^{(1)},\ldots,p_1^{(4)}$, which can be done in $n^4$ ways. Since now the differences $|q_{l}^{(i)} - p_{l}^{(i)}|$ are uniquely determined by the choice of the corresponding equivalence classes, we can proceed sequentially to see that there are at most two choices left for any pair $P_l^{(i)}$. To sum up, we have at most
\begin{equation*}
 2^{4k} \cdot n^4 \cdot n^r = C \cdot n^{r+4}
\end{equation*}
possibilities to choose $(P^{(1)},\ldots,P^{(4)})$. If now $r\leq 2k-2$, we can conclude that
\begin{equation}
 \# \mathcal{S}_n(\bpi) \leq C \cdot n^{2k+2}.
\label{eq3}
\end{equation}
Hence, it remains to consider the case where $r=2k-1$ and $r=2k$, respectively. \\

To begin with, let $r=2k-1$. Then, we have either two equivalence classes with three elements or one equivalence class with four. Since $\bpi$ is matched, there must exist an $i\in\left\{1,\ldots,4\right\}$ and an $l\in\left\{1,\ldots,k\right\}$ such that $P_l^{(i)}$ is not equivalent to any other pair in the sequence $P^{(i)}$. Without loss of generality, we can assume that $i=1$. In contrast to the construction of $(P^{(1)},\ldots,P^{(4)})$ as above, we now alter our procedure as follows: We fix all equivalence classes except of that $P_l^{(1)}$ belongs to. There are $n^{r-1}$ possibilities to accomplish that. Now we choose again one of $n^4$ possible values for $p_1^{(1)},\ldots,p_1^{(4)}$. Hereafter, we fix $q_m^{(1)}$, $m=1,\ldots,l-1$, and then start from $q_k^{(1)} = p_1^{(1)}$ to go backwards and obtain the values of $p_{k}^{(1)}, \ldots, p_{l+1}^{(1)}$. Each of these steps leaves at most two choices to us, that is $2^{k-1}$ choices in total. But now, $P_l^{(1)}$ is uniquely determined since $p_l^{(1)} = q_{l-1}^{(1)}$ and $q_l^{(1)} = p_{l+1}^{(1)}$ by consistency. Thus, we had to make one choice less than before, implying \eqref{eq3}. \\

Now, let $r=2k$. In this case, each equivalence class has exactly two elements. Since we consider a matched partition, we can find here as well an $l\in\left\{1,\ldots,k\right\}$ such that $P_l^{(1)}$ is not equivalent to any other pair in the sequence $P^{(1)}$. But in addition to that, we also have an $m\in\left\{1,\ldots,k\right\}$ such that, possibly after relabeling, $P_m^{(2)}$ is neither equivalent to any element in $P^{(1)}$ nor to any other element in $P^{(2)}$. Thus, we can use the same argument as before to see that this time, we can reduce the number of choices to at most $C \cdot n^{r+2} = C \cdot n^{2k+2}$. In conclusion, \eqref{eq3} holds for any matched partition $\bpi$. To sum up our results, we obtain that
\begin{align*}
& \E\left[\left(\tr\X_{n}^{k} - \E\left[\tr \X_{n}^{k}\right]\right)^4\right] \\
& \quad = \frac{1}{n^{2k}} \sum_{\stackrel{\bpi\in\mathcal{P}(4k),}{\bpi \ \text{matched}}} \sum_{(P^{(1)}, \ldots, P^{(4)})\in \mathcal{S}_n\left(\bpi\right)} \E\Big[\prod_{j=1}^{4} \left(a (P^{(j)}) - \E\left[a (P^{(j)})\right]\right)\Big] \leq C \cdot n^{2},
\end{align*}
which is the statement of Lemma~\ref{lemma}.

\end{proof}

From Lemma~\ref{lemma} and Chebyshev's inequality, we can now conclude that for any $\varepsilon>0$ and any $k,n\in\mathbb{N}$,
\begin{equation*}
 \mathbb{P}\left( \left| \frac{1}{n} \tr\X_n^k - \E \left[\frac{1}{n}\tr\X_n^k\right] \right| > \varepsilon \right) \leq \frac{C}{\varepsilon^4 n^2}.
\end{equation*}
Hence, the convergence in expectation part of Theorem~\ref{main} together with the Borel-Cantelli lemma yield that
\begin{equation*}
 \lim_{n\to\infty} \frac{1}{n} \tr\X_n^k = \E \left[Y^k\right] \qquad \text{almost surely},
\end{equation*}
where $Y$ is distributed according to the standard semi-circle law. In particular, we have that, with probability $1$, the empirical spectral distribution of $\X_n$ converges weakly to the semi-circle law.

\section{Examples}

\subsection{Gaussian processes} Let $\left\{a(p,p+r), p\in\mathbb{N}\right\}$, $r\in\N_0$, be independent families of stationary Gaussian Markov processes with mean $0$ and variance $1$. In addition to this, we assume that the processes are non-degenerate in the sense that $\E\left[a(p,p+r)|a(q,q+r),q\leq p-1\right] \neq a(p,p+r)$. In this case, the conditions of Theorem~\ref{main} are satisfied. Indeed, for fixed $r\in\N_0$ and any $p\in\N$, we can represent $a_p := a(p,p+r)$ as
\begin{equation*}
 a_p = x_p \sum_{j=1}^{p} y_j \xi_j,
\end{equation*}
where $\left\{\xi_j\right\}$ is a family of independent standard Gaussian variables and $x_p,y_1,\ldots,y_p \in \R\backslash\left\{0\right\}$. Then, we obtain
\begin{equation*}
 \cov(\tau) = \cov(a_p,a_{p+\tau}) = \frac{x_{p+\tau}}{x_p},
\end{equation*}
implying $\cov(\tau)=\cov(1)^\tau$ for any $\tau\in\N_0$. By calculating the second moment of $a_2 = x_2 y_2 \xi_2 + \cov(1) a_1$, we can conclude that $|\cov(1)|<1$. Thus, we have $\sum_{\tau=0}^{\infty} |\cov(\tau)|<\infty$.

\subsection{Markov chains with finite state space} We want to verify that condition (C$4$) holds for stationary $N$-state Markov chains which are ergodic and reversible. Let $\left\{X_k,k\in\N\right\}$ be such a Markov chain with mean $0$ and variance $1$. Denote by $P$ the corresponding $N\times N$ transition matrix and by $\pi$ its stationary distribution. Reversibility yields that $P$ is diagonalizable. Hence, for any $k\in\N$, we can write
\begin{equation*}
 P^k = TD^kT^{-1}
\end{equation*}
for some invertible matrix $T$ and a diagonal matrix $D=\mathrm{diag}(\la_1,\ldots,\la_N)$. Denoting by $s_1,\ldots,s_N$ the $N$ possible states of the chain, we get
\begin{align*}
 \cov(X_n,X_{n+k}) = \sum_{i,j=1}^{N} s_i s_j \pi(i) P^k(i,j) = \sum_{l=1}^{N} \la_l^k \left(\sum_{i=1}^{N} s_i \pi(i) T(i,l) \sum_{j=1}^{N} s_j T^{-1}(l,j) \right).
\end{align*}
Since $P$ is stochastic, we have $|\la_l|\leq 1$ for any $l=1,\ldots,N$. If $|\la_l| = 1$, ergodicity implies that $\la_l=1$. The corresponding space of right eigenvectors is spanned by the vector $v=(1,\ldots,1)^T$. Consequently,
\begin{equation*}
 \sum_{i=1}^{N} s_i \pi(i) T(i,l) = c \ \E\left[X_1\right] = 0.
\end{equation*}
Thus $\cov(X_n,X_{n+k})$ decays exponentially to $0$ as $k\to\infty$ and condition (C$4$) is satisfied.

\bibliographystyle{alpha}
\bibliography{gabi}

\end{document}